\documentclass[journal, twocolumn, 10pt]{IEEEtran}
\usepackage{amsfonts,amsmath,amssymb,amsthm}
\usepackage{algorithmic}
\usepackage{algorithm}
\usepackage{balance}
\usepackage{times}
\usepackage{bm}
\usepackage{graphicx}
\usepackage{setspace}
\usepackage{comment}
\usepackage{cases}
\graphicspath{{./fig_eps/}}
\usepackage{here}
\usepackage{ascmac} 
\usepackage{array}
\usepackage{float}
\usepackage{cite}
\usepackage{url}
\usepackage{subfigure}
\newtheorem{mydef}{Definition}
\newtheorem{mylem}{Lemma}
\newtheorem{mythm}{Theorem}
\newtheorem{myas}{Assumption}

\newtheorem{myrem}{Remark}

\newtheorem{mypro}{Problem}
\usepackage{balance}

\newcommand{\rfig}[1]{Fig.\,\ref{#1}} 
 
\newcommand{\req}[1]{\eqref{#1}} 
 
\newcommand{\rtab}[1]{Table\,\ref{#1}}
\newcommand{\rlem}[1]{Lemma\,\ref{#1}}
\newcommand{\rpro}[1]{Problem\,\ref{#1}}
\newcommand{\rrem}[1]{Remark\,\ref{#1}}
\newcommand{\rsec}[1]{Section\,\ref{#1}}
\newcommand{\rdef}[1]{Definition\,\ref{#1}}
\newcommand{\ras}[1]{Assumption\,\ref{#1}}

\newcommand{\qedwhite}{\hfill \ensuremath{\Box}}
\begin{document}
\title{Aperiodic Sampled-Data Control via Explicit Transmission Mapping: A Set Invariance Approach} 
\author{ Kazumune~Hashimoto,~\IEEEmembership{ Student Member,~IEEE},  Shuichi~Adachi,~\IEEEmembership{ Member,~IEEE,}  and~Dimos~V.~Dimarogonas,~\IEEEmembership{\large Member,~IEEE}
\thanks{Kazumune Hashimoto and Shuichi Adachi are with Department of Applied Physics and Physico-Informatics, Keio University, Japan. Hashimoto's work is supported by Grant-in-Aid for JSPS Research Fellow (DC2).}
\thanks{Dimos V. Dimarogonas is with the department of Electrical Engineering, KTH Royal Institute of Technology, Sweden. His work was supported by the Swedish Research Council (VR), the Swedish Foundation for Strategic Research (SSF), ERC Starting Grant BUCOPHSYS, and Knut och Alice Wallenberg foundation (KAW). 
}}
\maketitle
\begin{abstract}
Event-triggered and self-triggered control have been proposed in recent years as promising control strategies to reduce communication resources in Networked Control Systems (NCSs). Based on the notion of set-invariance theory, this note presents new self-triggered control strategies for linear discrete-time systems subject to input and state constraints. The proposed schemes not only achieve communication reduction for NCSs, but also ensure both asymptotic stability of the origin and constraint satisfactions. 
A numerical simulation example validates the effectiveness of the proposed approaches. 
\end{abstract}
\section{Introduction}
Efficient network utilization and energy-aware communication protocols between sensors, actuators and controllers have been recent challenges in the community of Networked Control Systems (NCSs). 
To tackle such challenges, event and self-triggered control schemes have been proposed as alternative approaches to the typical time-triggered controllers, see e.g., \cite{dimos2010a,heemels2011a,tabuada2010a}. In contrast to the time-triggered case where the control signals are executed periodically, event and self-triggered strategies trigger the executions based on the violation of prescribed control performances, such as Input-to-State Stability (ISS) \cite{dimos2010a} and ${\cal L}_\infty$ gain stability \cite{heemels2011a}. 

In particular, we are interested in designing self-triggered strategies for \textit{constrained} control systems, 
where certain constraints such as physical limitations and actuator saturations need to be explicitly taken into account. One of the most popular control schemes to deal with such constraints is Model Predictive Control (MPC) \cite{Mayne2000a}. In the MPC strategy, the current control action is determined by solving a constrained optimal control problem online, based on the knowledge of current state information and dynamics of the plant. Moreover, applications of the event and self-triggered control to MPC have been recently proposed to reduce the frequency of solving optimal control problems, see e.g.,  \cite{evmpc_linear6,evmpc_linear9,evmpc_linear10,hashimoto2015c,hashimoto2017a,hashimoto2017c}.

The main contribution of this note is to provide novel self-triggered strategies for constrained systems from an alternative perspective to the afore-cited papers, namely, a perspective from \textit{set-invariance theory} \cite{blanchini1999a}. Set invariance theory has been extensively studied for the past two decades  \cite{blanchini1994a,bitsoris1988a,gilbert1991a}, and it provides a fundamental tool to design controllers for constrained control systems. 
Two established concepts are those of a \textit{controlled invariant set} and $\lambda$-\textit{contractive set}. While a controlled invariant set implies that the state stays inside the set for all time, a $\lambda$-contractive set guarantees the more restrictive condition that the state is asymptotically stabilized to the origin. 

In this note, two different types of set-invariance based self-triggered strategies are presented. In the first approach, we formulate an optimal control problem such that the controller obtains stabilizing control inputs under \textit{multiple candidates} of transmission time intervals. 
Among the multiple solutions, the controller selects a suitable one such that both control performance and communication load are taken into account. 
Asymptotic stability of the origin is ensured by using Lyapunov techniques, where the Lyapunov function is induced by a  $\lambda$-contractive set obtained offline. Although the first approach guarantees asymptotic stability, it may lead to a high computation load as it requires to solve multiple optimization problems online. Therefore, we secondly propose an alternative strategy that aims to overcome the computational drawback of the first proposal. 
Similarly to the concept of explicit MPC \cite{bemporad2002b}, we provide an \textit{offline}, explicit mapping that sends the state information to the desired transmission time interval. As we will see in later sections, the state-space is decomposed into a finite number of subsets, to which appropriate transmission time intervals are assigned. 

The rest of the paper is organized as follows. In Section~II, the system description and some preliminaries of invariant set theory are given. In Section~III, we propose the first approach of the self-triggered strategy. In Section~IV, the second approach of the self-triggerd strategy is presented. In Section~V, a illustrative simulation example is given. We finally conclude in Section~VI. \\


\noindent
\textit{(Nomenclature)}: Let $\mathbb{R}$, $\mathbb{R}_+$, $\mathbb{N}$, $\mathbb{N}_+$ be the \textit{non-negative reals, positive reals, non-negative} and \textit{positive integers}, respectively. The \textit{interior} of the set ${\cal S} \subset \mathbb{R}^n$ is denoted as ${\rm int} \{ {\cal S} \}$. A set ${\cal S} \subset \mathbb{R}^n$ is called \textit{${\cal C}$-set} if it is compact, convex, and $0 \in {\rm int} \{{\cal S}\}$. 
For vectors $v_1, \ldots, v_N$, ${\rm co} \{ v_1, \ldots, v_N \}$ denotes their \textit{convex hull}. A set of vectors $\{v_1, \ldots, v_N \}$ whose convex hull gives a set ${\cal P}$ (i.e., ${\cal P} = {\rm co} \{ v_1, \ldots, v_N \}$), and each $v_n$, $n\in \{1, 2, \ldots, N\}$ is not contained in the convex hull of $v_1, \ldots, v_{n-1}, v_{n+1}, \ldots, v_{N}$ is called a set of \textit{vertices} of ${\cal P}$. Given a ${\cal C}$-set ${\cal S} \subset \mathbb{R}^n $, denote by $\partial {\cal S} \subset \mathbb{R}^n$ the boundary of ${\cal S}$. 
For a given $\lambda \in \mathbb{R}$ and a ${\cal C}$-set ${\cal S}\subset \mathbb{R}^n$, denote $\lambda {\cal S}$ as $\lambda {\cal S} = \{ \lambda x \in \mathbb{R}^n : x \in {\cal S}\}$. Given a set ${\cal S} \subset \mathbb{R}^n$, the function $\Psi_{\cal S} : \mathbb{R}^n \rightarrow \mathbb{R}_+$ with $\Psi_{{\cal S}}(x) = {\rm inf} \{\mu : x \in \mu {\cal S}, \mu \geq 0 \}$ is called a \textit{gauge function}. 
For given two sets ${\cal S}_1, {\cal S}_2 \subset \mathbb{R}^n$, define ${\cal S}_1 \backslash {\cal S}_2$
as ${\cal S}_1 \backslash {\cal S}_2 = \{ x\in \mathbb{R}^n : x \in {\cal S}_1, x \notin {\cal S}_2 \}$. 
\section{Problem formulation and some preliminaries}\label{strategy_sec}
In this section, the system description and some established results of set-invariance theory are provided.  
\subsection{System description and control strategy}\label{sys_desc_sec}
Consider a networked control system illustrated in 
\rfig{network}. We assume that the dynamics of the plant are given by 
\begin{equation}\label{sys}
x ({k+1}) = A x (k) + Bu (k)
\end{equation}
for $k\in \mathbb{N}$, where $x (k) \in \mathbb{R}^{n}$ is the state and $u (k) \in \mathbb{R}^m$ is the control variable.
The state and control input are assumed to be constrained as $x(k) \in {\cal X},\ u(k) \in {\cal U}$, $\forall k \in \mathbb{N}$, where ${\cal X} \subset \mathbb{R}^n, \ {\cal U} \subset \mathbb{R}^m$ are both polyhedral ${\cal C}$-sets described as 
\begin{equation}\label{constraint}
\begin{array}{lll}
{\cal X} = \{x\in \mathbb{R}^n: H_x x\leq h_x \}, \\
{\cal U} = \{u \in \mathbb{R}^m :H_u u\leq h_u \}, 
\end{array}
\end{equation}
where $H_x \in \mathbb{R}^{n_x \times n}$, $H_u \in \mathbb{R}^{n_u \times m}$ and $h_x$, $h_u$ are appropriately sized vectors having positive components. 
The control objective is to steer the state to the origin, i.e., $x(k) \rightarrow 0$ as $k\rightarrow \infty$. Let $k_m$, $m\in\mathbb{N}$ with $k_0 = 0$ be the transmission time instants when the plant transmits the state information $x(k_m)$ to the controller and updates the control input. 
In the self-triggered strategy, the transmission times are determined as
\begin{equation}\label{transmission_times}
k_{m +1} = k_m + \Gamma (x(k_m)), \ \ m \in \mathbb{N},
\end{equation}
where $\Gamma : {\cal X} \rightarrow \{ 1, 2, \ldots, j_{\max} \}$ denotes a mapping that sends the state information to the corresponding transmission time interval. Here, a maximal transmission time interval $j _{\max} \in \mathbb{N}_+$ is set apriori in order to formulate the self-triggered strategy. Due to the limited nature of communication bandwidth, we assume that only one control sample (not a sequence of control samples) is allowed to be transmitted at each transmission time. Namely, the control input is constant between two consecutive inter-transmission times, i.e., 
\begin{equation}\label{controller}
u (k) = \kappa (x(k_m)) \in {\cal U}, \ \ k \in [k_m , k_{m+1} ), 
\end{equation}
where $\kappa : {\cal X} \rightarrow {\cal U}$ denotes the state-feedback control law. 
The following assumptions are made throughout the paper (see e.g., \cite{blanchini1994a}):  
\begin{myas} 
The pair $(A, B)$ is controllable.
\end{myas}
\begin{myas}
{The matrix $B$ has full column rank.}
\end{myas}

\begin{figure}[tbp]
  \begin{center}
   \includegraphics[width=7cm]{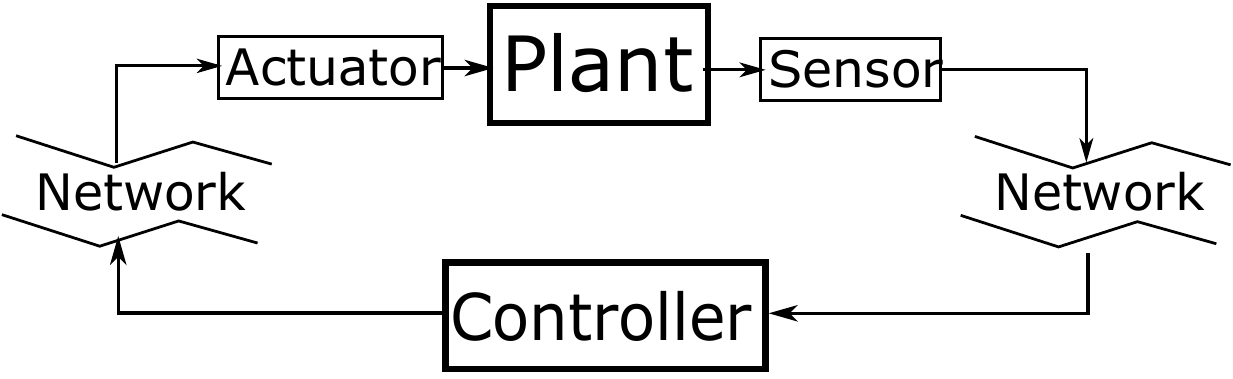}
   \caption{Networked Control System}
   \label{network}
  \end{center}
 \end{figure}

\subsection{Set-invariance theory}\label{set_invariance_sec}
In the following, we define the standard notions of \textit{controlled invariant set} and $\lambda$\textit{-contractive set} \cite{blanchini1994a}, which are  important concepts to characterize invariance and convergence properties for constrained control systems.
\begin{mydef}[Controlled invariant set]\label{lambda_contractive}
For a given ${\cal C}$-set ${\cal S} \subseteq {\cal X}$,  ${\cal S}$ is said to be a controlled invariant set in ${\cal X}$, if and only if there exists a control law $g(x)\in {\cal U}$ such that $Ax + Bg(x) \in {\cal S}$ for all $x\in {\cal S}$. 
\end{mydef}
\begin{mydef}[$\lambda$-contractive set]\label{lambda_contractive}
For a given ${\cal C}$-set ${\cal S} \subseteq {\cal X}$,  ${\cal S}$ is said to be a $\lambda$-contractive set in ${\cal X}$ for $\lambda\in [0, 1] $, if and only if there exists a control law $g(x)\in {\cal U}$ such that $A x + B g(x) \in \lambda {\cal S}$ for all $x\in {\cal S}$.
\end{mydef}
Roughly speaking, a set ${\cal S}$ is called $\lambda$-contractive set if all states in ${\cal S}$ can be driven into a tighter (or equivalent) region $\lambda {\cal S}$ by applying a one-step control input. From the definition, a controlled invariant set implies a $\lambda$-contractive set with $\lambda = 1$. 
We review several established results for obtaining a contractive set and the corresponding properties. For given $\lambda \in [0, 1)$ and ${\cal C}$-set ${\cal X}\subset \mathbb{R}^n$, there are several ways to efficiently construct a $\lambda$-contractive set in ${\cal X}$. 
For a given ${\cal C}$-set ${\cal D} \subset \mathbb{R}^n$, let ${\cal Q}_{\lambda}: \mathbb{R}^n \rightarrow \mathbb{R}^n$ be the mapping 
\begin{equation}\label{kstep_contrl}
{\cal Q}_{\lambda} ({\cal D}) = \{ x\in {\cal X} : \exists u \in {\cal U},\  A x + B u \in \lambda {\cal D} \}. 
\end{equation}
A simple algorithm to obtain a $\lambda$-contractive set in ${\cal X}$ is to compute $\Omega_{j} \subset \mathbb{R}^n$, $j \in \mathbb{N}$ as
\begin{equation}\label{iterative_procedure}
\Omega_{0} = {\cal X}, \quad \Omega_{j+1}  = {\cal Q} _{\lambda } (\Omega_{j} )  \cap {\cal X}, 
\end{equation}
and then it holds that the set ${\cal S} = \lim_{j\rightarrow \infty} \Omega_j $ is $\lambda$-contractive, see e.g., \cite{blanchini1994a}. 
If $\Omega_{j+1} =\Omega_j$ for some $j$, the $\lambda$-contractive set is obtained as ${\cal S} = \Omega_j $, which requires only a finite number of iterations. Although such condition does not hold in general, it is still shown, under Assumptions~1 and 2, that the algorithm converges in the sense that for every $\lambda < \bar{\lambda} <1$, there exists a finite $j \in \mathbb{N}_+$ such that the set $\Omega_j $ is $\bar{\lambda}$-contractive (see \textit{Theorem 3.2} in \cite{blanchini1994a}). 
Several other algorithms have been recently proposed, see e.g., \cite{hovd2014a,hovd2016a} and see also \cite{darup2017a} for a detailed convergence analysis. 
The following lemma illustrates the existence of a (non-quadratic) Lyapunov function in a given $\lambda$-contractive set:
\begin{mylem}\label{stability_lem}{\normalfont \cite{blanchini1994a}:}
Let ${\cal S} \subset {\cal X}$ be a $\lambda$-contractive ${\cal C}$-set with $\lambda \in [0, 1]$ and the associated gauge function $\Psi_{\cal S} : {\cal S} \rightarrow \mathbb{R}_+$. Then, there exists a control law $g : {\cal X} \rightarrow {\cal U}$ such that 
\begin{equation}\label{set_induced_lyapunov}
\Psi_{\cal S} ( Ax + B g (x)) \leq \lambda \Psi_{\cal S} ( x),
\end{equation}
for all $x \in {\cal S}$.
\end{mylem}
\rlem{stability_lem} follows immediately from \rdef{lambda_contractive}.
If $\lambda < 1$, \req{set_induced_lyapunov} implies the existence of a stabilizing controller in ${\cal S}$ in the sense that the output of the gauge function $\Psi_{\cal S} (\cdot)$ is guaranteed to decrease. The gauge function $\Psi_{\cal S} (\cdot)$ defined in ${\cal S}$ is known as \textit{set-induced Lyapunov function} in the literature; for a detailed discussion, see e.g., \cite{blanchini1994a}. 

\section{Self-triggered strategy}
As described in the introduction, we propose two different types of self-triggered controllers; in this section, the first approach is presented. 
\subsection{Designing a stabilizing  controller}\label{stabilize_control_sec}
For a given $\lambda \in [0, 1)$, let us first construct a $\lambda$-contractive set ${\cal S}$ in ${\cal X}$. 
Note that since ${\cal X}$ is a polyhedral ${\cal C}$-set, one can efficiently compute the $\lambda$-contractive set through polyhedral operations according to \req{iterative_procedure} \footnote{If the iterative procedure in \req{iterative_procedure} does not converge in finite time, one can stop the procedure to obtain a $\bar{\lambda}$-contractive set ($\lambda < \bar{\lambda} <1$) in a finite number of iterations. In such case, we use $\bar{\lambda}$ (instead of $\lambda$) as the parameter to design the self-triggered strategies provided throughout the paper. 
}. 
The obtained $\lambda$-contractive set ${\cal S}$ can be denoted as 
\begin{equation}\label{setS}
{\cal S} = {\rm co} \{v_1, v_2, \ldots, v_N \} \subseteq {\cal X}, 
\end{equation}
where $v_n, n\in \{1, 2, \ldots, N \}$ represent the vertices of ${\cal S}$, and $N$ represents the number of them. 
\begin{myas}\label{initial_cond}
The initial state is inside ${\cal S}$, i.e., $x(k_0) \in {\cal S}$. 
\end{myas}
Based on \ras{initial_cond}, we will design the self-triggered strategy such that the state remains in ${\cal S}$ for each transmission time instant. 
Suppose that at a certain transmission time $k_m$, $m \in \mathbb{N}$, the plant transmits the state information $x(k_m)$ to the controller. 
Based on $x(k_m)$, the controller needs to compute both a suitable controller to be applied and a transmission time interval, such that the state is stabilized to the origin. To this end, we first propose an approach to obtain stabilizing controllers under \textit{multiple candidates} of transmission time intervals. More specifically, we obtain different control actions under different transmission time intervals, and the controller selects a suitable one among them. 
To obtain the stabilizing controllers, we formulate the following optimal control problem for each $j\in \{1, \ldots, j_{\max} \}$: 

\begin{mypro}[Optimal Control Problem for $j$]\label{control_problem}
\normalfont
For given $x(k_m)$,
$j \in \{1, \ldots, j_{\max} \}$ and the $\lambda$-contractive set ${\cal S}$, find $u \in {\cal U}$ and $\varepsilon \in \mathbb{R}$ by solving the following problem:\\

\noindent
\begin{equation}\label{cost_epsilon}
\underset{u \in {\cal U}} {\min}\ \ \varepsilon, 
\end{equation}

\noindent
subject to $\varepsilon \in [0,  \lambda]$, and
\begin{enumerate}
\renewcommand{\labelenumi}{(C.\arabic{enumi})}
\item $A^{j'} x(k_m) + \sum^{j'} _{i=1} A^{i-1} B u \in {\cal X}$, $\forall j' \in \{ 1, \ldots, j \}$; 
\item $A^j x(k_m) + \sum^{j} _{i=1} A^{i-1} B u \in \varepsilon \varepsilon_{x} {\cal S}$;
\end{enumerate} 
where $\varepsilon_x = \Psi_{\cal S}(x(k_m))$.
\qedwhite
\end{mypro}
In (C.2), $A^j x(k_m) + \sum^{j} _{i=1} A^{i-1} B u$ represents a state by applying the control input $u \in {\cal U}$ \textit{constantly} for $j$ time steps. Moreover, from the definition of the gauge function $\Psi_{\cal S}(\cdot)$ we have $x(k_m) \in \varepsilon_{x} {\cal S}$. 
Thus, \rpro{control_problem} aims to find the smallest possible scaled set $\varepsilon \varepsilon_{x} {\cal S}$, such that the state enters $\varepsilon \varepsilon_{x} {\cal S}$ (from $\varepsilon_x {\cal S}$) by applying a $j$-step constant control input. This means that a stabilizing controller is found under the transmission time interval $j$. The constraint in (C.1) implies that the state must remain inside ${\cal X}$ while applying a $j$-step constant controller, which is imposed to guarantee the constraint satisfaction. 
Note that \rpro{control_problem} is a linear program, since all constraints imposed in (C.1), (C.2), as well as the cost in \req{cost_epsilon} are all linear. 

For given $x(k_m)$ and $j$, let $(u^* _j, \varepsilon^* _j )$ be a pair of optimal solutions obtained by solving \rpro{control_problem}. 
From (C.2), the state enters $\varepsilon^* _j \varepsilon_x {\cal S}$ if $u^* _j$ is applied constantly for $j$ steps, i.e., 
$A^j x(k_m) + \sum^{j} _{i=1} A^{i-1} B u^* _j \in \varepsilon^* _j \varepsilon _{x} {\cal S}$, which means that we have $\Psi_{\cal S} (A^j x(k_m) + \sum^{j} _{i=1} A^{i-1} B u^* _j ) \leq \varepsilon^* _j \Psi_{\cal S} (x(k_m))$, or 
\begin{equation}\label{lyapunov2}
\begin{aligned}
 \Psi_{\cal S} (A^j x(k_m) + & \sum^{j} _{i=1} A^{i-1} B u^* _j )  -\Psi_{\cal S} (x(k_m)) \\
& \leq - (1-\varepsilon^* _j) \Psi_{\cal S} (x(k_m))   
\end{aligned}
\end{equation}
with $0 \leq \varepsilon^* _j \leq \lambda < 1$.
Thus, $1-\varepsilon^* _j$ represents how much the output of the gauge function (as a Lyapunov function candidate) decreases by applying the optimal controller $u^* _j$ constantly for $j$ steps. That is, if $1-\varepsilon^* _j$ becomes larger (i.e., $\varepsilon^* _j$ becomes smaller), then the state will be closer to the origin and a better control performance is achieved. 

Now, consider solving \rpro{control_problem} for all $j \in \{1, \ldots, j_{\max} \}$, which provides different solutions under different transmission time intervals. 
In the following, let ${\cal J} (x(k_m))$ be the set of indices (transmission time intervals) where \rpro{control_problem} provides a feasible solution. 
That is, 
\begin{equation}\label{calJ}
\begin{aligned}
{\cal J} (x(k_m)) = \{j \in & \{1,   \ldots, \  j_{\max}\} : \\
                       & {\rm \rpro{control_problem}\ is\ feasible\ for\ } j   \}. 
\end{aligned}
\end{equation}

\begin{myrem}[On the non-emptiness of ${\cal J}(x(k_m))$] \label{feasibility_problem}
\normalfont 
If $x(k_m) \in {\cal S}$, there always exists $u \in {\cal U}$ such that $A x(k_m) + B u \in \lambda \varepsilon_x {\cal S} \subseteq {\cal X}$ holds from the properties of the $\lambda$-contractive set (see \cite{blanchini1994a}). Thus, \rpro{control_problem} has a solution with $j=1$ for any $x(k_m) \in {\cal S}$, and hence, ${\cal J} (x(k_m))$ is non-empty for any $x(k_m) \in {\cal S}$. \qedwhite
\end{myrem}

\subsection{An overall algorithm}\label{transmission_sec}
In this subsection an overall self-triggered algorithm is  presented. After solving \rpro{control_problem} for all $j \in \{1,\ldots, j_{\max} \}$, which provides the optimal (feasible) sets of solutions $(u^* _j, \varepsilon^* _j )$ for all $j \in {\cal J} (x(k_m))$, the controller selects a suitable transmission time interval among them. 
The transmission time interval is selected such that both control performance and the communication load are taken into account. A more specific way to achieve this is given in the following overall strategy: \\ 

\textit{{Algorithm~1}} (Self-triggered strategy): 
For any transmission time $k_m$, $m \in \mathbb{N}$, do the following: 
\begin{enumerate}
\item The plant transmits the current state information $x(k_m)$ to the controller. 
\item Based on $x(k_m)$, the controller solves Problem~1 for all $j \in \{1, \ldots, j_{\max} \}$, which provides the optimal (feasible) solutions $(\varepsilon^* _j, u^* _j)$ for all $j \in {\cal J} (x(k_m))$. 
\item The controller picks up an optimal index $j_m \in {\cal J} (x(k_m))$ by solving the following problem: 
\begin{equation}\label{cost_func}
j_m = \underset{j \in {\cal J}(x(k_m))}{\rm argmax}\  w_1 (1-\varepsilon^* _j)/j + w_2 j , 
\end{equation}
where $w_1, w_2 \geq 0$ represent given tuning weight parameters. Then, set $k_{m+1} = k_m +  j_m$ and $u^* (k_m) = u^* _{j_m}$,  and the controller transmits $u^* (k_m)$ and $k_{m+1}$ to the plant. 
\item The plant applies $u^* (k_m)$ for all $k \in [k_m, k_{m+1})$. Set $m \leftarrow m+1$, and then go back to step (1). \qedwhite 
\end{enumerate}

As shown in Algorithm~1, for each $k_m$ we select the transmission time interval $j_m$ according to \req{cost_func}. As described in the previous subsection, the term $(1-\varepsilon^* _j)$ represents how much the output of the gauge function decreases by applying the optimal controller $u^* _j$ constantly for $j$ steps. Thus, the first term $(1-\varepsilon^* _j)/j$ represents a \textit{reward} due to the rate of decrease of the gauge function per {one time step}, and a better control performance can be achieved when this term becomes larger. 
On the other hand, from a self-triggered control viewpoint, less control updates will be obtained when control inputs can be applied constantly longer (i.e., when $j$ becomes larger). Thus, the second part in \req{cost_func} involves $j$ to represent some {reward} for alleviating the communication load; as $j$ gets larger, then we obtain less communication load and a larger reward is obtained. 

Some remarks are in order regarding Algorithm~1: 
\begin{myrem}[Relation to move-blocking MPC]
\normalfont
The proposed algorithm is related to \textit{move-blocking MPC}\cite{move_blocking}, in the sense that the optimal control inputs are restricted to be constant for some time period. Note that move-blocking MPC aims at reducing the computational complexity by decreasing the degrees of freedom of the optimal control problem\cite{move_blocking}; the proposed approach, on the other hand, aims at reducing the \textit{communication load} through the move-blocking technique, and the reduction of computation load is not a primary objective here.  \qedwhite 
\end{myrem}
\begin{myrem}[On the selection of $j_{\max}$]
\normalfont
In Algorithm~1, the controller solves \rpro{control_problem} for all $j \in \{ 1, \ldots, j_{\max} \}$ for each transmission time instant. While we can potentially achieve longer transmission intervals if $j_{\max}$ is selected larger, the computation load of solving \rpro{control_problem} becomes heavier. 
Thus, in practical implementation, the user may carefully select a suitable $j_{\max}$ by considering the trade-off between the communication load and the calculation time of solving the optimal control problem. 
\qedwhite 
\end{myrem}

\begin{mythm}[Stability]\label{stability1}
Suppose that Assumption~3 holds, and Algorithm~1 is implemented. Then, it holds that $x (k) \rightarrow 0$ as $k \rightarrow \infty$. \qedwhite
\end{mythm}
\begin{proof}
We first show that \req{cost_func} is always feasible (i.e., we can always pick up a transmission time interval according to \req{cost_func}), by proving that ${\cal J} (x(k_m))$ is non-empty for all $m\in \mathbb{N}$. 
By \ras{initial_cond} we obtain $x(k_0) \in {\cal S}$ and thus ${\cal J} (x(k_0))$ is non-empty (see \rrem{feasibility_problem}). Since $j_0$ is obtained from \req{cost_func}, we have $j_0 \in {\cal J} (x(k_0))$ which means that \rpro{control_problem} has a feasible solution for $j = j_0$. Thus, from the constraint (C.2) in \rpro{control_problem}, we obtain $x (k_1) = A^{j_0} x(k_0) + \sum^{j_0} _{i=1} A^{i-1} B u^* (k_0) \in {\cal S}$, which means that ${\cal J} (x(k_1))$ is non-empty. By recursively following this argument, it is shown that $x (k_m) \in {\cal S}$ for all $m\in \mathbb{N}$, which follows that ${\cal J} (x(k_m))$ is non-empty for all $m\in \mathbb{N}$. 

Now, it is shown that $x (k) \rightarrow 0$ as $k \rightarrow \infty$. Since $j_m \in {\cal J} (x_m)$, $\forall m\in \mathbb{N}$, it holds from (C.2) in \rpro{control_problem} that:  
\begin{equation}\label{gauge_decrease}
\begin{aligned}
\Psi_{\cal S} (x(k_{m+1}) ) & \leq \varepsilon^* _{j_m} \Psi_{\cal S} (x(k_m))  \\
& \leq \lambda \ \Psi_{\cal S} (x(k_m)).
\end{aligned}
\end{equation}
with $\lambda < 1$. Therefore, by regarding $\Psi_{\cal S} (\cdot)$ as a set-induced Lyapunov function candidate (see \rlem{stability_lem}), the Lyapunov function is strictly decreasing and the state trajectory is asymptotically stabilized to the origin. This completes the proof.
\end{proof}

\begin{myrem}[On achieving exponential stability]
\normalfont 
Although Theorem~1 states only asymptotic stability of the origin, exponential stability can be achieved by imposing an additional constraint when evaluating the reward function in \req{cost_func}. Specifically, let $j_\ell$ be chosen according to \req{cost_func}, subject to the constraint $\varepsilon^* _{j} \leq \lambda^{j}$. 
Indeed, imposing this constraint yields that $\Psi_{\cal S} (x(k_{m+1}) ) \leq \lambda^{j_{m}} \Psi_{\cal S} (x(k_{m}))$, $\forall m \in \mathbb{N}$ (instead of $\Psi_{\cal S} (x(k_{m+1}) ) \leq \lambda \Psi_{\cal S} (x(k_{m}))$ as in \req{gauge_decrease}). Thus, we obtain 
\begin{equation*}\label{exponential_stability}
\Psi_{\cal S} (x(k_{m}) ) \leq \lambda^{j_{m-1}} \Psi_{\cal S} (x(k_{m-1})) \leq \cdots \leq \lambda^{k_m} \Psi_{\cal S} (x(0)), 
\end{equation*}
which implies that exponential stability is guaranteed (see e.g., \cite{blanchini1994a}). 

\qedwhite 
\end{myrem}

\section{Self-triggered control via explicit mapping $\Gamma$}
In the previous section, the self-triggered strategy has been presented by solving \rpro{control_problem} for all $j \in \{ 1, \ldots, j_{\max} \}$. However, solving \rpro{control_problem} for all candidates of transmission time intervals may lead to a high computation load, which may induce computational delays to transmit control samples to the plant. A more preferred approach may be that the transmission mapping $\Gamma: {\cal X} \rightarrow \{1, \ldots, j_{\max}\}$ given in \req{transmission_times}, which sends the state to the desired transmission time interval, is obtained \textit{offline}. That is, with the mapping $\Gamma$ provided explicitly offline, the next transmission time can be directly determined from the (current) state information, without having to solve Problem 1 for all $j \in \{ 1, \ldots, j_{\max} \}$. The approach presented in this section is related to explicit MPC framework \cite{bemporad2002b}, in which an offline characterization of the control strategy (but here, the transmission time intervals) is given via state-space decomposition. A more specific formulation is given below. 
\subsection{Construction of $\Gamma$ via state-space decomposition}
In order to create the explicit mapping of $\Gamma$, we first \textit{decompose} the contractive set ${\cal S}$ into a finite number of disjoint subsets ${\cal S}_1, \ldots, {\cal S}_L \subset {\cal S}$, i.e., 
\begin{equation}\label{decomposition_subsets}
{\cal S} = \bigcup^{L} _{\ell =1} {\cal S}_\ell, 
\end{equation}
where it holds that ${\cal S}_\ell \cap {\cal S}_{\ell'} = \emptyset$ for all $(\ell, \ell') \in \{1, \ldots, L\} \times \{1, \ldots, L\}$ $(\ell \neq \ell')$. Based on the decomposition, we will then assign a specific transmission time interval to each ${\cal S}_\ell$, $\ell \in \{1, \ldots, L \}$, so that the controller directly determines the next transmission time. Intuitively, if the state is located far from the origin we would like to assign a short transmission time interval to achieve stability of the origin (or achieve good control performance). In particular, in the case of un-stable systems, applying a constant control signal may lead to a divergence of states, especially if the state is far from the origin. On the other hand, if the state is close to the origin, a small control effort may be sufficient to stabilize the system. That is, assigning a long transmission time interval may be allowable to achieve both stability and communication reduction. 

Motivated by the above intuition, we decompose the contractive set as follows. First, for a given $L \in \mathbb{N}_+$, define a set of scalars $\rho_1, \ldots \rho_L \in (0, 1]$, with 
\begin{equation}
0 < \rho_1 < \rho_2 < \cdots < \rho_{L-1} < \rho_{L} = 1.
\end{equation}
Then, consider the following sequence of $L$ sets ${\cal S}_1, {\cal S}_2, \ldots , {\cal S}_{L} \subset {\cal S}$: 
\begin{equation}\label{set_sequence}
\begin{aligned}
{\cal S}_1    &= \rho_1 {\cal S}, \\
{\cal S}_\ell & = \rho_\ell  {\cal S} \backslash {\cal S}_{\ell-1},\ \ \forall \ell \in \{2, \ldots, L\}. 
\end{aligned}
\end{equation}

\begin{figure}[tbp]
  \begin{center}
   \includegraphics[width=5cm]{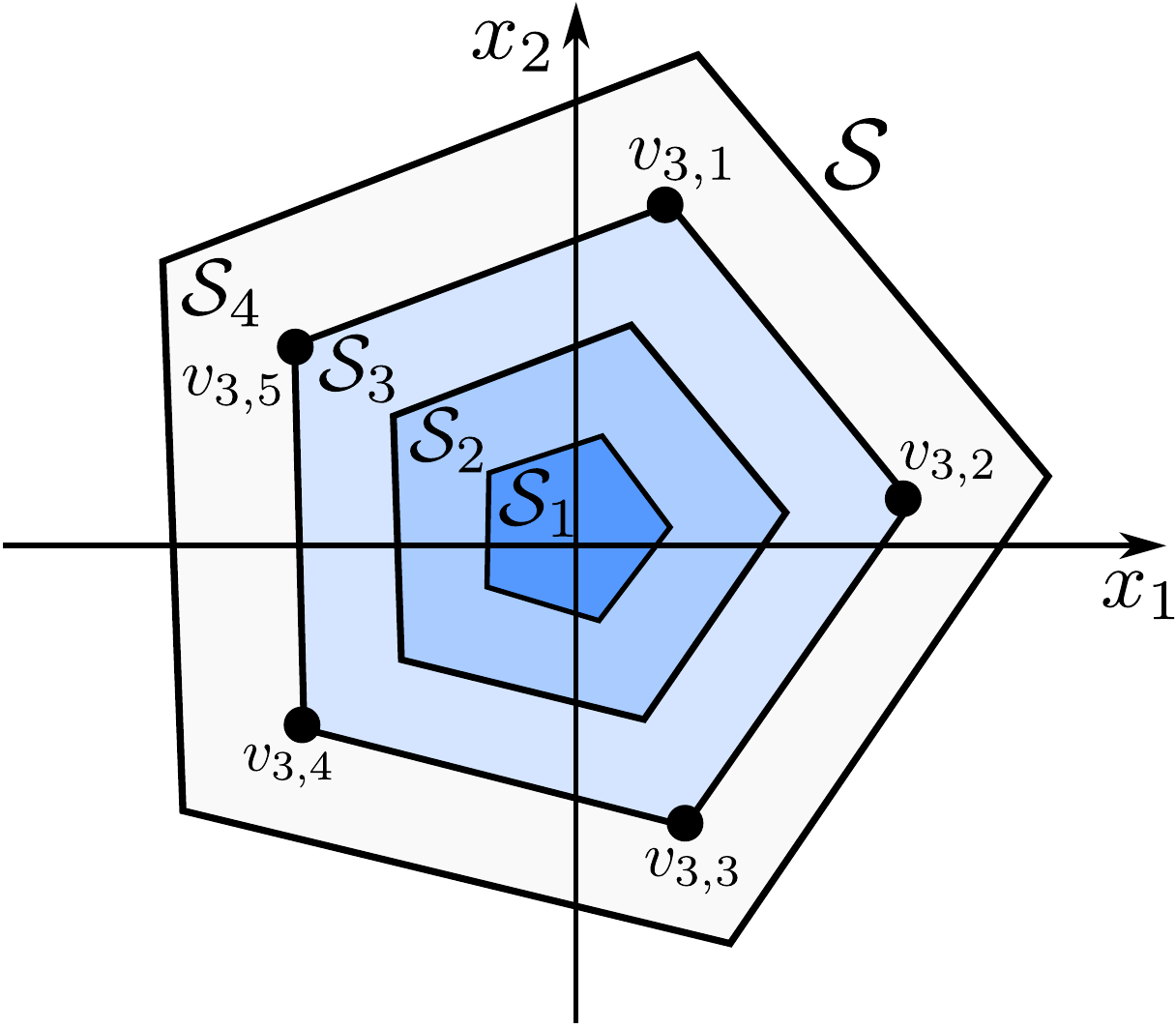}
   \caption{Illustration of the decomposed subsets ${\cal S}_\ell$, $\ell \in \{1, \ldots, L\}$ (for the case $L=4$) according to \req{set_sequence}. In the figure, the pentagon represents the contractive set ${\cal S}$ and the subsets ${\cal S}_1, \ldots, {\cal S}_4$ are illustrated with different shades of blue. } 
   \label{decomposition}
  \end{center}
 \end{figure}

The illustration of the sequence of sets is depicted in \rfig{decomposition}. It can be easily shown that the set sequence defined in \req{set_sequence} yields a {decomposition} as in \req{decomposition_subsets}, which satisfies the disjoint property as described above. 
Now, ${\cal S}$ has been decomposed into a finite number of $L$ sets ${\cal S}_1, \ldots, {\cal S}_L$, to which we next assign suitable transmission time intervals. To this end, we formulate the following optimal control problem for each pair $(\ell, j) \in \{1, \ldots, L \} \times \{1, \ldots, j_{\max}\}$: 
\begin{mypro}[Optimal Control Problem for $(\ell, j)$] \label{control_problem2}
\normalfont  
For a given pair $(\ell, j) \in \{1, \ldots, L\}\times \{1, \ldots, j_{\max} \}$, find $u_{1}, \ldots, u_{N} \in {\cal U}$ and $\varepsilon \in \mathbb{R}$ by solving the following problem: \\

\noindent
\begin{equation}
\underset{u_1, \ldots, u_{N} \in {\cal U}} {\min}\ \ \varepsilon, 
\end{equation}

\noindent
subject to $\varepsilon \in [0, \lambda]$, and 
\begin{enumerate}
\renewcommand{\labelenumi}{(C.\arabic{enumi})}
\item For all $n \in \{ 1, \ldots, N \}$, $j' \in \{ 1, \ldots, j \}$, 
\begin{eqnarray}\label{state_constraint}
A^{j'} v_{\ell, n} + \sum^{j'} _{i=1} A^{i-1} B u_n \in {\cal X}, 
\end{eqnarray}
where $v_{\ell, n} = \rho_\ell v_{n}$, $n \in \{ 1, \ldots, N \}$, $j' \in \{ 1, \ldots, j \}$. 
\item For all $n \in \{ 1, \ldots, N \}$, 
\begin{equation}\label{stability_constraint}
\begin{aligned}
& A^j v_{\ell, n} + \sum^{j} _{i=1} A^{i-1} B u_n \in \varepsilon \rho_\ell {\cal S}. 
\end{aligned}
\end{equation}
\end{enumerate}
\qedwhite
\end{mypro}
Recall that $v_n$, $n \in \{1, \ldots, N\}$ represent the vertices of ${\cal S}$ (see \req{setS}). Thus, $v_{\ell, n} = \rho_\ell v_n$, $n \in \{1, \ldots, N\}$ represent the extreme points on the outer boundary of ${\cal S}_\ell$ (see the illustration in \rfig{decomposition}).  \rpro{control_problem2} for $(\ell, j)$ aims at finding a set of controllers $u_1, \ldots, u_{N}$ and a scalar $\varepsilon$, such that all the extreme points $v_{\ell, n}$, $n \in \{1, \ldots, N\}$ can be driven into $\varepsilon \rho_\ell {\cal S}$ under the $j$-step constant control inputs. Note that \rpro{control_problem2} is solved {offline} for all $\ell \in \{1,\ldots, L\}$, $j\in \{1, \ldots, j_{\max} \}$, since it can be solved by evaluating the extreme points $v_{\ell, n}$, $n \in \{1, \ldots, N\}$ that are given offline. 

Now, suppose that \rpro{control_problem2} has a solution for $(\ell, j)$, which provides optimal control inputs and a scalar denoted as $u^* _{\ell, j}(1), u^* _{\ell, j}(2), \ldots, u^* _{\ell, j} (N) \in {\cal U}$, $\varepsilon^* _{\ell, j} \in [0, \lambda]$, respectively. The following lemma describes that the feasibility of \rpro{control_problem2} for $(\ell, j)$ implies the existence of a stabilizing controller for all $x \in {\cal S}_{\ell}$: 

\begin{mylem}\label{stabilize_controller}
\normalfont 
Suppose that \rpro{control_problem2} finds a solution for $(\ell, j)$. 
Then, for every $x \in {\cal S}_\ell$, there exists $u \in {\cal U}$ such that: (i)~$A^{j'} x + \sum^{j'} _{i=1} A^{i-1} B u \in {\cal X}$ for all $j' \in \{ 1, \ldots, j \}$; (ii)~$A^j x + \sum^{j} _{i=1} A^{i-1} B u \in \varepsilon^* _{\ell, j} \varepsilon_x \ {\cal S}$ with $\varepsilon_x = \Psi_{\cal S} (x)$. 
\qedwhite
\end{mylem}

\begin{proof}
Since \rpro{control_problem2} has a solution for $(\ell, j)$, from \req{state_constraint} and \req{stability_constraint} we obtain 
$A^{j'} v_{\ell, n} + \sum^{j'} _{i=1} A^{i-1} B u^* _{\ell, j} (n) \in {\cal X}$ 
for all $n \in \{ 1, \ldots, N \}$, $j' \in \{ 1, \ldots, j \}$, and 
$A^j v_{\ell, n} + \sum^{j} _{i=1} A^{i-1} B u^* _{\ell, j} (n) \in \varepsilon^* _{\ell, j} \rho_\ell {\cal S}$ 
for all $n \in \{1, \ldots, N \}$. 

Suppose $x \in {\cal S}_\ell$ and let $\varepsilon_x = \Psi_{{\cal S}} (x) \in [0, 1]$. Since $x \in {\cal S}_\ell \subseteq \rho_\ell {\cal S}$, we have $\varepsilon_x \leq \rho_\ell$. Moreover, since $x \in \varepsilon_x {\cal S}$, there exist $\lambda_n \in [0, 1]$, $n\in \{1, \ldots, N \}$ such that $x = \varepsilon_x  \sum^{N} _{n=1} \lambda_n v_{n} = (\varepsilon_x/\rho_\ell) \sum^{N} _{n=1} \lambda_n v_{\ell, n}$, $\sum^{N} _{n=1} \lambda_n = 1$, where we have used $v_{\ell, n} = \rho_\ell v_n$. Let $u \in \mathbb{R}^m$ be given by 
\begin{equation}\label{feasible_input}
u = \frac{\varepsilon_x}{\rho_\ell} \sum^{N} _{n=1} \lambda_n u^* _{\ell, j} (n) \in {\cal U}. 
\end{equation}
Then, for all $j' \in \{1, \ldots, j \}$, we obtain 
\begin{equation*}
\begin{aligned}
 & A^{j'} x + \sum^{j'} _{i=1} A^{i-1} B u \\
     &= \frac{\varepsilon_x}{\rho_\ell}  \sum^{N} _{n=1} \lambda_n ( A^{j'} v_{\ell, n} + \sum^{j'} _{i=1} A^{i-1} B  u^* _{\ell, j} (n)) \in \frac{\varepsilon_x}{\rho_\ell}  {\cal X} \subseteq {\cal X},
\end{aligned}
\end{equation*}
where the first inclusion holds since $A^{j'} v_{\ell, n} + \sum^{j'} _{i=1} A^{i-1} B u^* _{\ell, j} (n) \in {\cal X}$ for all $n \in \{ 1, \ldots, N \}$, $j' \in \{ 1, \ldots, j \}$, and the last inclusion holds since $\varepsilon_x \leq \rho_\ell$. 
Moreover, we have
\begin{equation*}
\begin{aligned}
 & A^j x + \sum^j _{i=1} A^{i-1} B u \\
     &= \frac{\varepsilon_x}{\rho_\ell}  \sum^{N} _{n=1} \lambda_n ( A^j v_{\ell, n} + \sum^j _{i=1} A^{i-1} B  u^* _{\ell, j} (n)) \\
     & \in \frac{\varepsilon_x}{\rho_\ell}\ \rho_\ell \ \varepsilon^* _{\ell, j} {\cal S} =\varepsilon^* _{\ell, j} \varepsilon_x {\cal S}, 
\end{aligned}
\end{equation*}
where the inclusion holds since $A^j v_{\ell, n} + \sum^{j} _{i=1} A^{i-1} B u^* _{\ell, j} (n) \in \rho_\ell \varepsilon^* _{\ell, j}  {\cal S}$ for all $n \in \{1, \ldots, N \}$. Hence, we obtain $A^j x + \sum^{j} _{i=1} A^{i-1} B u \in \varepsilon^* _{\ell, j} \varepsilon_x \ {\cal S}$. 
This completes the proof. 
\end{proof}
\rlem{stabilize_controller} implies that for every $x \in {\cal S}_\ell$ there exists $u\in{\cal U}$ such that $\Psi_{\cal S} (A^j x + \sum^j _{i=1} A^{i-1} B u) \leq \varepsilon^* _{\ell, j} \Psi_{\cal S} (x)$ holds. 
Thus, this means that if $x \in {\cal S}_\ell$, then there exists a $j$-step stabilizing controller such that the output of the gauge function decreases. 
As mentioned previously in \rsec{stabilize_control_sec}, $(1-\varepsilon^* _{\ell, j})$ represents the decreasing rate of the gauge function. Thus, we can evaluate the control performance by $\varepsilon^* _{\ell, j}$ similarly to the self-triggered strategy presented in the previous section. 

Now, suppose that for each $\ell$ we solve \rpro{control_problem2} for all $j \in \{1, \ldots, j_{\max} \}$. 
Let ${\cal J}_{\ell}$ be a set of indices (transmission time intervals) where \rpro{control_problem2} has a feasible solution for $\ell$, i.e., 
\begin{equation}\label{calJ2}
\begin{aligned}
{\cal J}_{\ell} = \{j \in  & \{1,  \ldots, \  j_{\max}\} : \\
                       & {\rm \rpro{control_problem2}\ is\ feasible\ for\ } (\ell, \ j) \}. 
\end{aligned}
\end{equation}
Regarding the feasible set ${\cal J}_\ell$, we obtain the following:
\begin{mylem}\label{feasible_lem2}
${\cal J}_{\ell}$ is non-empty for all $\ell \in \{ 1, \ldots, L \}$. 
\end{mylem}
The proof immediately follows from \rdef{lambda_contractive} and is given in the Appendix. By evaluating the feasible solutions obtained above, we now assign to each ${\cal S}_\ell$ a suitable transmission time interval. In Algorithm~1, we presented the self-triggered strategy by determining the transmission interval according to the reward function in \req{cost_func}. Motivated by this, we similarly now consider the following assignment of the transmission time interval to ${\cal S}_\ell$, by taking both control performance and communication load into account: 
\begin{equation}\label{cost_func2}
j^* _{\ell} = \underset{j \in {\cal J}_{\ell} }{\rm argmax}\ w_1  (1-\varepsilon^* _{\ell, j})/j + w_2  j, 
\end{equation}
where $w_1, w_2 \geq 0$ denote the given tuning weights associated to each part of the reward similarly to \req{cost_func}. Note that in contrast to the previous self-triggered strategy where a suitable transmission time interval is obtained online, \req{cost_func2} is now given in an offline fashion. 
Suppose that we compute $j^* _\ell$ according to \req{cost_func2} for all $\ell \in \{1, \ldots, L \}$. Then, each $j^* _{\ell}$ is assigned to ${\cal S}_{\ell}$ as the suitable transmission time interval. That is, if $x(k_m) \in {\cal S}_{\ell}$ for a certain transmission time $k_m$,  the controller directly sets the next transmission time as $k_{m+1} = k_m + j^* _{\ell}$. Let ${\cal T} : {\cal S} \rightarrow \{1, \ldots, j_{\max} \}$ be a mapping from ${\cal S}_\ell$ to the assigned transmission time interval, i.e., $ j^* _\ell = {\cal T} ({\cal S}_\ell )$. Moreover, let ${\cal R}: {\cal X} \rightarrow {\cal S}$ be a mapping from $x$ to the corresponding subset that $x$ belongs to, i.e., 
${\cal R} (x) =  {\cal S}_\ell,  \ \  {\rm iff} \ x\in {\cal S}_\ell,\ \ell \in \{1, \ldots, L\}. $
Then, the overall transmission mapping $\Gamma: {\cal X} \rightarrow \{1, \ldots, j_{\max} \}$ is given by 
\begin{equation}\label{gamma}
\Gamma (x) =  ({\cal T} \circ {\cal R}) (x).
\end{equation} 

\subsection{An overall algorithm}
Given the explicit transmission mapping obtained in \req{gamma}, the overall self-triggered algorithm is now provided below: \\ 


\textit{Algorithm~2 (Self-triggered strategy via explicit mapping $\Gamma$)}: 
Given the explicit mapping $\Gamma$ obtained by \req{gamma} and for any transmission time $k_m$, $m \in \mathbb{N}$, do the following: 
\begin{enumerate}
\item The plant transmits the current state information $x(k_m)$ to the controller. 
\item Based on $x(k_m)$, the controller sets the transmission time interval as $j_m =\Gamma (x(k_m))$. Then, set the next trasmission time as $k_{m+1} = k_m + j_m$. 
\item Suppose that $x(k_m) \in {\cal S}_{\ell_m}$ for some $\ell_m \in \{1, \ldots, L \}$. For a given $j_m$ obtained in step 2), the controller sets $u^* (k_m) = (\varepsilon_x / \rho_{\ell_m}) \sum^{N} _{n=1} \lambda_n u^* _{\ell_m, j_m} (n) \in {\cal U}$, where $\varepsilon_x = \Psi_{\cal S}(x(k_m))$ and $u^* _{\ell_m, j_m} (1), \ldots, u^* _{\ell_m, j_m} (N) \in {\cal U}$ are the solution to \rpro{control_problem2} for $(\ell_m, j_m)$. Then, the controller transmits $u^* (k_m)$ and $k_{m+1}$ to the plant. 
\item The plant applies $u^* (k_m)$ for all $k \in [k_m, k_{m+1})$. Set $m \leftarrow m+1$ and then go back to step 1). \qedwhite
\end{enumerate}

As shown in Algorithm~2, in contrast to the first approach the controller only needs to compute the control input for a given transmission time interval from the explicit mapping $\Gamma$. 
\begin{myrem}[The point location problem]
\normalfont 
For each transmission time $k_m$, the controller needs to find a suitable subset ${\cal S}_\ell$ such that $x(k_m)\in {\cal S}_\ell$ holds to determine the assigned transmission time interval according to \req{gamma}. This problem, which we call \textit{the point location problem}, can be easily solved by using the following property: we have $x(k_m) \in {\cal S}_1 \Leftrightarrow x(k_m) \in \rho_1 {\cal S}$, and for all $\ell \in \{2, \ldots, L\}$,
\begin{equation}
x(k_m) \in {\cal S}_\ell\ \Leftrightarrow\ x(k_m) \notin \rho_{\ell-1} {\cal S},\ x(k_m) \in \rho_{\ell} {\cal S}. 
\end{equation}
Hence, the point location problem can be solved by checking if $x(k_m) \in \rho_\ell {\cal S}$, $\ell \in \{1, \ldots, L\}$ sequentially in that order, and takes the first index $\ell$ such that $x (k_m) \in \rho_\ell {\cal S}$ holds. 
\qedwhite 
\end{myrem}
\begin{mythm}[Stability]
Suppose that Assumptions~3 holds, and Algorithm~2 is implemented. Then, it holds that $x (k) \rightarrow 0$ as $k \rightarrow \infty$. \qedwhite 
\end{mythm}
\begin{proof}
We first show that selecting $j_m$ as $j_m = \Gamma (x (k_m) )$ is always feasible by proving that $x(k_m) \in {\cal S}$ for all $m \in \mathbb{N}$ (if $x(k_m) \notin {\cal S}$, the controller cannot determine $j_m$ since the mapping is not defined). By \ras{initial_cond}, we obtain $x(k_0) \in {\cal S}$. To prove by induction, assume $x(k_m) \in {\cal S}$ for some $m\in \mathbb{N}_+$, and we will show $x(k_{m+1}) \in {\cal S}$. Suppose that $x(k_m) \in {\cal S}_{\ell_m} \subseteq {\cal S}$ for some $\ell_m \in \{ 1,\ldots, L\}$, which means from \req{cost_func2} that $j_m = j^* _\ell = \Gamma (x (k_m))$ with $\ell = \ell_m$. Since $j_m = j^* _\ell \in {\cal J}_\ell$ with $\ell = \ell_m$, \rpro{control_problem2} has a solution for the pair $(\ell_m, j_m)$. Let $\varepsilon^* _{\ell_m, j_m} \in [0, \lambda]$ be the optimal $\varepsilon$ as a solution to \rpro{control_problem2} for $(\ell_m,  j_m)$. Then, from the proof of \rlem{stabilize_controller}, setting $u^*(k_m) = (\varepsilon_x /\rho_{\ell_m}) \sum^{N} _{n=1} \lambda_n u^* _{\ell_m, j_m} (n)$ yields that $x(k_{m+1}) = A^{j_m} x(k_m) + \sum^{j_m} _{i=1} A^{i-1} B u^* (k_m) \in \varepsilon_x \varepsilon^* _{\ell_m, j_m} {\cal S} \subseteq {\cal S}$. Therefore, we have $x(k_m) \in {\cal S}$ for all $m\in \mathbb{N}$. 


Now, it is shown that $x (k) \rightarrow 0$ as $k \rightarrow \infty$. Since $x(k_m) \in {\cal S}$ for all $m\in \mathbb{N}$, we obtain from \rlem{stabilize_controller} that 
\begin{equation}
\begin{aligned}
\Psi_{\cal S} (x(k_{m+1}) ) & \leq \varepsilon^* _{\ell_m, j_m} \Psi_{\cal S} (x(k_m))  \\
& \leq \lambda \ \Psi_{\cal S} (x(k_m)), 
\end{aligned}
\end{equation}
with $\lambda < 1$. Therefore, by considering $\Psi_{\cal S}(\cdot)$ as a set-induced Lyapunov function candidate, the state trajectory is asymptotically stabilized to the origin. This completes the proof.
\end{proof}

\subsection{Comparisons between first and second approach} \label{discuss_alg2_sec}
In this subsection we discuss both advantages and drawbacks of the second approach (Algorithm~2), by making some comparisons with the first one (Algorithm~1). As stated previously, the second approach is advantageous over the first one in terms of the computation load, since the transmission mapping $\Gamma$ is given offline according to the procedure presented in the previous subsection. Note, however, that in the second approach, each $j^* _\ell$ is computed by solving \rpro{control_problem2} that evaluates \textit{the extreme points} of ${\cal S}_\ell$ (i.e., $v_{\ell, n}$, $n \in \{1, \ldots, N\}$). This means that, while $x(k_m)$ is in the \textit{interior} of ${\cal S}_\ell$, which is not on some extreme point of ${\cal S}_\ell$, there may exist some $j \in {\cal J}_\ell \ (j\neq j^* _\ell)$, such that applying $u^* (k_m) = (\varepsilon_x/\rho_\ell)  \sum^{N} _{n=1} \lambda_n u^* _{\ell, j} (n) \in {\cal U}$ {could} yield a \textit{larger} reward in \req{cost_func2} than the one obtained with $\varepsilon^* _{\ell, j^* _\ell}$. In this sense, the second approach yields a suboptimal (or conservative) solution compared with Algorithm~1 on the selection of transmission time intervals. This observation is also illustrated in the simulation example, where it is shown that Algorithm~1 achieves less communication load than Algorithm~2 for the case $w_1 = 0$ (for details, see Section~V). 

\section{Simulation results}
In this section we provide an illustrative example to validate our control schemes. The simulation was conducted on Matlab 2016a under Windows 10, Intel(R) Core(TM) 2.40 GHz, 8 GB RAM, using Multi-Parametric Toolbox (MPT3) to compute the $\lambda$-contractive set. We consider a control problem of a \textit{batch reactor system}, which is often utilized as a benchmark in the NCSs community (see, e.g., \cite{heemels2012a}). The linearized model is given in the continuous-time domain as $\dot{x} (t) = A_c x(t) + B_c u(t)$, where $A_c, B_c$ are given by 
\begin{equation*}
{\small 
\begin{aligned}
A_c & =  \left [
\begin{array}{cccc}
1.380  &  -0.208 & 6.715   & -5.676  \\
-0.581 & -4.290  & 0        & 0.675 \\
1.067   & 4.273   & -6.654  & 5.893 \\
0.048   & 4.273   & 1.343  & -2.104 \\
\end{array}
\right ] \\ 
B_c & = \left [
\begin{array}{cc}
0       &    0 \\
5.679  &    0  \\
1.136   & -3.146 \\
1.136   & 0
\end{array}
\right ]. 
\end{aligned}
}
\end{equation*}
The system is unstable having unstable poles $1.9911, 0.0633$. We assume ${\cal X} = \{ x\in \mathbb{R}^4 : ||x||_\infty \leq 2 \}$, ${\cal U} = \{ u \in \mathbb{R}^2 : ||u||_{\infty} \leq 5 \}$ and $j_{\max} = 30$. We obtain the corresponding discrete-time system under a zero-order-hold controller with a sampling time interval $0.1$, and the $\lambda$-contractive set ${\cal S}$ is obtained with $\lambda=0.99$ according to the procedure presented in Section II. \rfig{state_trajectory} illustrates the resulting state trajectories and the corresponding control inputs by implementing Algorithm~1, starting from the initial state $x(k_0) = [1;\ 2;\ 2;\ 0.5]$ and the weights $(w_1, w_2) = (50, 1)$. The figure shows that the resulting state trajectories are asymptotically stabilized to the origin, and control inputs are updated only when necessary. 

To analyse the effect of weights, we again simulate Algorithm~1 with $x(k_0) = [1;\ 2;\ 2;\ 0.5]$ under different selection of weights $(w_1, w_2) =(0, 1), (50,1), (100, 1)$. We then compute the convergence time steps when the state enters the small region around the origin (the region satisfying $||x|| \leq 0.001$), and the total number of transmission instances during the time period $k \in [0, 100]$. The results are shown in \rtab{result_alg1}. From the table, $(w_1, w_2) = (100, 1)$ achieves the fastest speed of  convergence. This is due to the fact that by selecting $w_1$ larger, the reward for the control performance (i.e., the first term in \req{cost_func}) is emphasized to be obtained. On the other hand, the number of transmission instances is the smallest for the case $(w_1, w_2) =(0, 1)$, which means that the smallest communication load is obtained. Therefore, it is shown that there exists a trade-off between control performance and communication load, and such trade-off can be regulated by tuning the weights $(w_1, w_2)$. 

\begin{table}[b]
\begin{center}
\caption{Convergence time and number of transmission instances} \label{result_alg1}
\begin{tabular}{cccc} \hline 
 ($w_1, w_2$) & (0,1) & $(50,1)$ & $(100, 1)$  \\ \hline \hline
Convergence (steps) & 141 &  93     &  69  \\ \hline
Transmission instances & 5 &  6    &  10 \\ \hline \hline
\end{tabular}
\end{center}
\end{table}

\begin{figure}[tbp]
  \begin{center}
   \includegraphics[width=7.8cm]{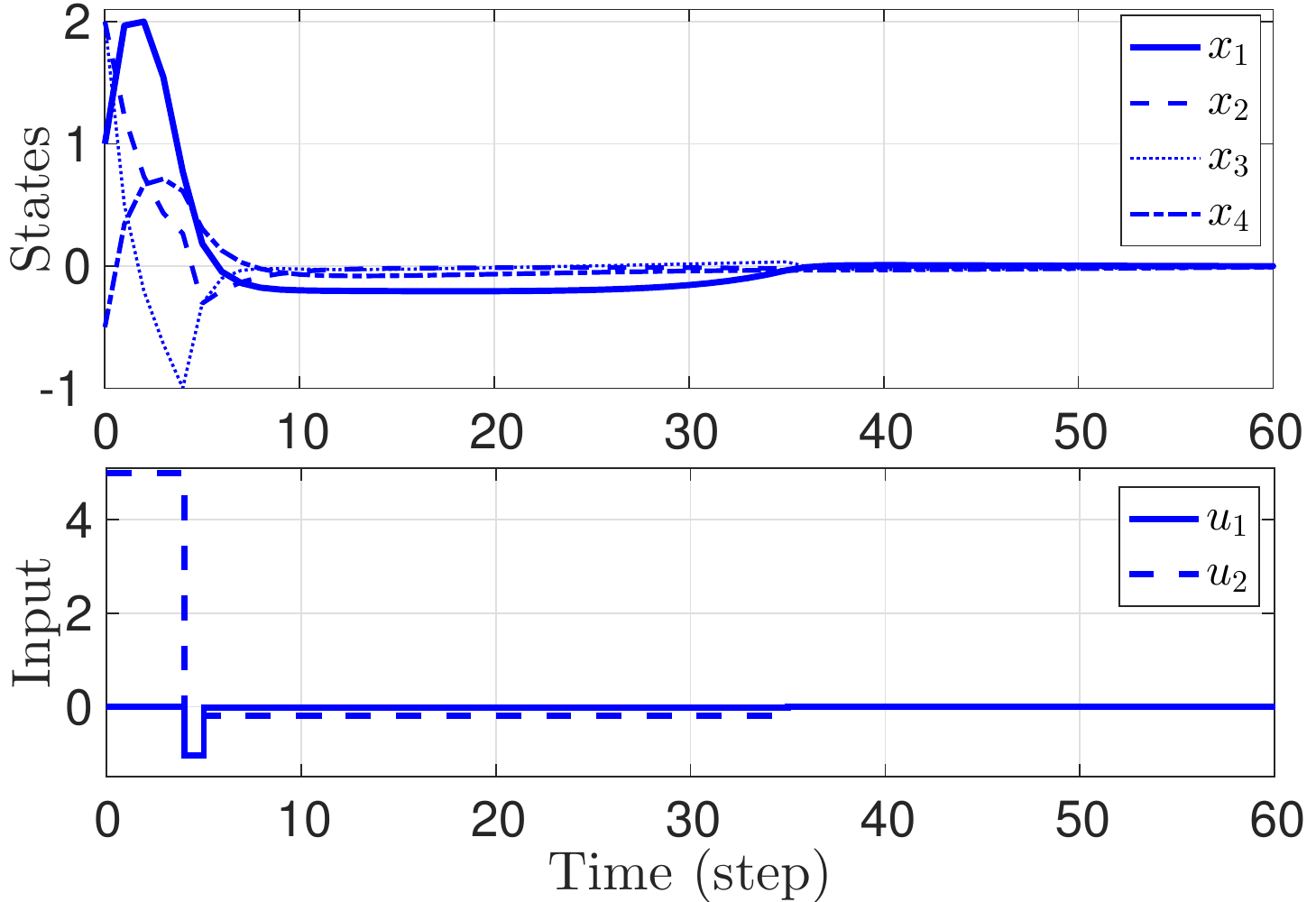}
   \caption{Simulation results of state trajectories (upper) and the control inputs (lower) by applying Algorithm~1.}
   \label{state_trajectory}
  \end{center}
 \end{figure}
 
 \begin{figure}[tbp]
  \begin{center}
   \includegraphics[width=7.8cm]{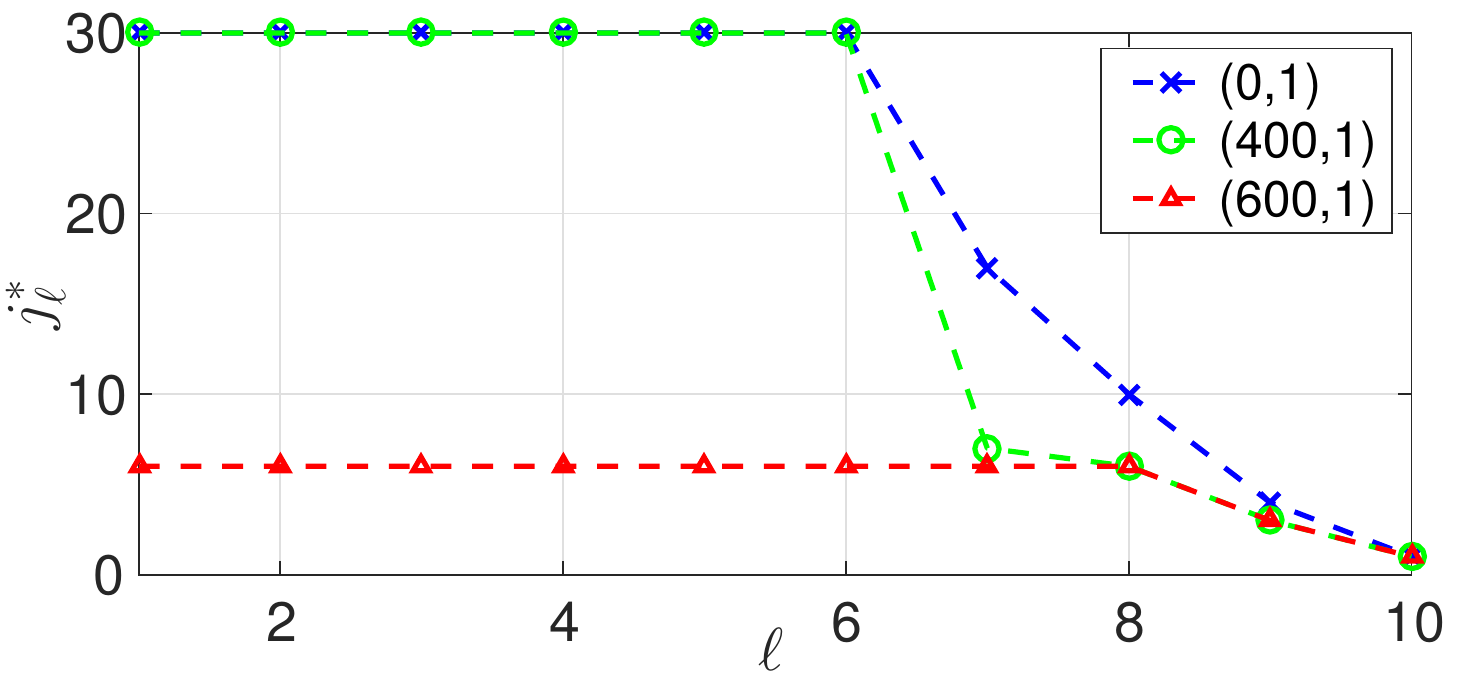}
   \caption{$j^* _\ell$ as a function of $\ell$ with $(w_1, w_2) = (0, 1), (400, 1), (600,1)$. }
   \label{jmax}
  \end{center}
 \end{figure}

To implement the second proposal, we decompose ${\cal S}$ into $L=10$ subsets with $\rho_\ell = 0.1 \ell$, $ \ell \in \{1, \ldots, 10 \}$ and the mapping $\Gamma$ is constructed according to Section IV. The selected transmission time intervals $j^* _\ell$ as a function of $\ell$ are illustrated in \rfig{jmax} under different selections of the weights $(w_1, w_2) = (0, 1), (400, 1), (600,1)$. The figure shows that the transmission time interval tends to be smaller as the weight $w_1$ increases. This means that attaining control performance is emphasized more than attaining communication reduction, if the weight for achieving the control performance $w_1$ is selected larger. In \rfig{jmax}, we also illustrate the selected transmission time intervals for $(w_1, w_2) = ( 0 , 1 )$ (i.e., the blue dashed line). Since $w_1 = 0$, each $j^* _\ell$ corresponds to the \textit{largest} transmission time interval such that \rpro{control_problem2} becomes feasible for $\ell$ (i.e., the maximal index in the feasible set ${\cal J}_\ell$). 
As shown in the figure, the feasible transmission time interval gets smaller as $\ell$ increases. Intuitively, this is due to that for unstable systems, applying a constant control input leads to a divergence of states, especially if the extreme points to solve \rpro{control_problem2} are located far from the origin. 


To illustrate the calculation time and the optimality of Algorithm~1 and 2 as discussed in \rsec{discuss_alg2_sec}, we again simulate the two algorithms with $(w_1, w_2) = (0, 1)$ and the initial state $x(k_0) = [1;\ 2;\ 2;\ 0.5]$. As previously described, setting $w_1 = 0$ corresponds to selecting the largest index in ${\cal J}_\ell$. 
\rtab{result_tab} illustrates the total number of transmission instances and the average calculation time to compute the control input for each transmission instance (i.e., the calculation time from step 2) to step 3) in Algorithms~1, 2). Here, both the total number of transmission instances and the average calculation time are computed over the time period $k \in [0, 100 ]$.  
From the table, Algorithm~1 achieves less communication load than Algorithm~2. As already discussed in \rsec{discuss_alg2_sec}, this is because of the sub-optimality of Algorithm~2; while Algorithm~1 solves \rpro{control_problem} based on the current state information online, \rpro{control_problem2} is solved offline by evaluating the extreme points of the subsets. On the other hand, Algorithm~2 achieves less calculation time than Algorithm~1, as the transmission mapping $\Gamma$ is explicitly given offline. 

\begin{table}[tbp]
\begin{center}
\caption{Number of transmission instances and average calculation time with $(w_1, w_2) = (0, 1)$. } \label{result_tab}
\begin{tabular}{ccc} \hline 
                             & Algorithm 1    &  Algorithm 2  \\ \hline \hline
Transmission instances &  5     &  7  \\ \hline
Calculation time (sec)      &  2.98     &  0.83 \\ \hline\hline
\end{tabular}
\end{center}
\end{table}

\section{Conclusion}
In this note, we present two different types of self-triggered strategies based on the notion of set-invariance theory. In the first approach, we formulate an optimal control problem such that suitable transmission time intervals are selected by evaluating both the control performance and the communication load. 
The second approach aims to overcome the computation drawback of the first one by providing an offline characterization of the mapping $\Gamma$. 
 In this approach, the state space is decomposed into a finite number of subsets, to which suitable transmission time intervals are assigned. Finally, the proposed self-triggered strategies are illustrated through a numerical example of controlling a batch reactor system.

\appendix
\textit{(Proof of \rlem{feasible_lem2}) : } 
We show that ${\cal J}_\ell$ is non-empty for all $\ell \in \{1, \ldots, L \}$, by proving that \rpro{control_problem2} is feasible for $(\ell, 1)$. Since $v_{\ell, 1}, \cdots, v_{\ell, N} \in \rho_\ell {\cal S}$, it holds that there exist a set of controllers $\tilde{u}_{\ell, 1}, \cdots, \tilde{u}_{\ell, N}\in {\cal U}$, such that $A v_{\ell, n} + B \tilde{u}_{\ell, n} \in \lambda \rho_\ell {\cal S} \subseteq {\cal X}$, $\forall n \in \{1, \ldots, N \}$ from the properties of the $\lambda$-contractive set. Thus, this directly means from \req{state_constraint}, \req{stability_constraint} that 
\rpro{control_problem2} has a feasible solution for $(\ell, 1)$, with $\varepsilon = \lambda$ and $u_{\ell, n} = \tilde{u}_{\ell, n}$, $\forall n \in \{1, \ldots, N \}$. This completes the proof. \qedwhite


{\small 
\begin{thebibliography}{10}

\bibitem{dimos2010a}
A.~Eqtami, D.~V. Dimarogonas, and K.~J. Kyriakopoulos, ``Event-triggered
  control for discrete time systems,'' in \emph{Proceedings of American Control
  Conference (ACC)}, 2010, pp. 4719--4724.

\bibitem{heemels2011a}
M.~C.~F. Donkers and W.~P. M.~H. Heemels, ``Output-based event-triggered
  control with guaranteed {${\mathcal L}_{\infty}$} gain and decentralized
  event-triggering,'' \emph{IEEE Transactions on Automatic Control}, vol.~57,
  no.~6, pp. 1362--1376, 2011.

\bibitem{tabuada2010a}
A.~Anta and P.~Tabuada, ``To sample or not to sample: Self-triggered control
  for nonlinear systems,'' \emph{IEEE Transactions on Automatic Control},
  vol.~55, no.~9, pp. 2030--2042, 2010.

\bibitem{Mayne2000a}
D.~Q. Mayne, J.~B. Rawlings, C.~V. Rao, and P.~O.~M. Scokaert, ``Constrained
  model predictive control: Stability and optimality,'' \emph{Automatica},
  vol.~36, pp. 789--814, 2000.

\bibitem{evmpc_linear6}
E.~Henriksson, D.~E. Quevedo, E.~G.~W. Peters, H.~Sandberg, and K.~H.
  Johansson, ``Multiple-loop self-triggered model predictive control for
  network scheduling and control,'' \emph{IEEE Transactions on Control Systems
  Technology}, vol.~23, no.~6, pp. 2167--2181, 2015.

\bibitem{evmpc_linear9}
K.~Kobayashi and K.~Hiraishi, ``Self-triggered model predictive control with
  delay compensation for networked control systems,'' in \emph{Proceedings of
  the 38th Annual Conference of the IEEE Industrial Electronics Society}, 2012,
  pp. 3182--3187.

\bibitem{evmpc_linear10}
F.~D. Brunner, W.~P. M.~H. Heemels, and F.~Allg{\"o}wer, ``Robust
  self-triggered mpc for constrained linear systems: A tube-based approach,''
  \emph{Automatica}, vol.~72, pp. 73--83, 2016.

\bibitem{hashimoto2015c}
K.~Hashimoto, S.~Adachi, and D.~V. Dimarogonas, ``Distributed aperiodic model
  predictive control for multi-agent systems,'' \emph{IET Control Theory and
  Applications}, vol.~9, no.~1, pp. 11--20, 2015.

\bibitem{hashimoto2017a}
------, ``Self-triggered model predictive control for nonlinear input-affine
  dynamical systems via adaptive control samples selection,'' \emph{IEEE
  Transactions on Automatic Control}, vol.~62, no.~1, pp. 177--189, 2017.

\bibitem{hashimoto2017c}
------, ``Event-triggered intermittent sampling for nonlinear model predictive
  control,'' \emph{Automatica}, vol.~81, pp. 148--155, 2017.

\bibitem{blanchini1999a}
F.~Blanchini, ``Set invariance in control,'' \emph{Automatica}, vol.~35,
  no.~11, pp. 1747--1767, 1999.

\bibitem{blanchini1994a}
------, ``Ultimate boundedness control for uncertain discrete-time systems via
  set-induced lyapunov functions,'' \emph{IEEE Transactions on Automatic
  Control}, vol.~39, no.~2, pp. 428--433, 1994.

\bibitem{bitsoris1988a}
G.~Bitsoris, ``Positively invariant polyhedral sets of discrete-time linear
  systems,'' \emph{International Journal of Control}, vol.~47, no.~6, pp.
  1713--1726, 1988.

\bibitem{gilbert1991a}
E.~G. Gilbert and K.~T. Tan, ``Linear systems with state and control
  constraints: The theory and application of maximal output admissible sets,''
  \emph{IEEE Transactions on Automatic Control}, vol.~36, no.~9, pp.
  1008--1020, 1991.

\bibitem{bemporad2002b}
{A. Bemporad, M. Morari, V. Dua and E. N. Pistikopoulos}, ``The explicit linear
  quadratic regulator for constrained systems,'' \emph{Automatica}, vol.~38,
  no.~1, pp. 3--20, 2002.

\bibitem{hovd2014a}
M.~Hovd, S.~Olaru, and G.~Bitsoris, ``Low complexity constraint control using
  contractive sets,'' in \emph{The IFAC World Congress}, 2014, pp. 2933--2938.

\bibitem{hovd2016a}
S.~Munir, M.~Hovd, G.~Sandou, and S.~Olaru, ``Controlled contractive sets for
  low-complexity constrained control,'' in \emph{Proceedings of 2016 IEEE
  Conference on Computer Aided Control System Design (Part of 2016 IEEE
  Multi-Conference on Systems and Control)}, 2016, pp. 856--861.

\bibitem{darup2017a}
M.~S. Darup and M.~Cannon, ``On the computation of lambda-contractive sets for
  linear constrained systems,'' \emph{IEEE Transactions on Automatic Control},
  vol.~62, no.~3, pp. 1498--1504, 2017.

\bibitem{move_blocking}
R.~Cagienard, P.~Grieder, E.~C. Kerrigan, and M.~Morari, ``{Move blocking
  strategies in receding horizon control},'' \emph{{Journal of Process
  Control}}, vol.~17, no.~6, pp. 563--570, 2007.

\bibitem{heemels2012a}
W.~P. M.~H. Heemels, K.~H. Johansson, and P.~Tabuada, ``An introduction to
  event-triggered and self-triggered control,'' in \emph{Proceedings of the
  51st IEEE Conference on Decision and Control (IEEE CDC)}, 2012, pp.
  3270--3285.

\end{thebibliography}
}
\end{document}